\documentclass{article}%
\usepackage[utf8]{inputenc}
\usepackage{amsfonts}
\usepackage{amssymb}
\usepackage{amsmath}
\usepackage{amsthm}
\usepackage{fullpage}
\usepackage{xcolor}
\providecommand{\U}[1]{\protect\rule{.1in}{.1in}}
\allowdisplaybreaks
\newtheorem{thm}{Theorem}
\newtheorem{lem}{Lemma}

\newtheorem{exa}{Example}

\newtheorem{rem}{Remark}

\newtheorem{cor}{Corollary}

\newtheorem{defn}{Definition}

\begin{document}

\title{ Generalized Balancing and Balancing-Lucas numbers }
\author{Hasan Al-Zoubi and  Ala'a Al-Kateeb\footnote{
Yarmouk university, Irbid-Jordan, alaa.kateeb@yu.edu.jo.}}

\maketitle
\begin{abstract}In this paper, we introduce a generalization of Balancing and Balancing-Lucas numbers. We describe some of their  properties also we give the related matrix
representation and divisibility properties. 
\end{abstract}



\section{Introduction} 
In mathematics, the Lucas sequences $U_{n}(P,Q) $ and $V_{n}(P,Q)$ are certain constant-recursive integer sequences that satisfy the recurrence relation
\[ x_{n}=P\cdot x_{n-1}-Q\cdot x_{n-2}\]
where $P$ and $Q$ are fixed integers. Any sequence satisfying this recurrence relation can be represented as a linear combination of the Lucas sequences $U_{n}(P,Q) $ and $V_{n}(P,Q)$.
Famous examples of Lucas sequences include the Fibonacci numbers, Mersenne, Pell, Lucas, Jacobsthal and, balancing numbers. Lucas sequences are named after the French mathematician E. Lucas. The sequences mentioned above satisfy many common properties and identities for
example Binet formulas, Catalan identities and matrix representation. Recently,
these sequences were investigated,  generalized or extended for example see \cite{AK, SW,FA,YY}.\\ When $P=6$ and $Q=1$, the sequence $x_n$ is called the sequence of Balancing numbers and denoted by $B_n$, that is \[B_n=6B_{n-1}-B_{n-2}, B_0=0, B_1=1\] 
Balancing numbers have several nice properties for example $8B_n^2+1$ is a perfect square for any integer $n$. Also, it can be easily proved that $B_n=\frac{\alpha^n-\beta^n}{\alpha-\beta}$, where $\alpha=3+2\sqrt{2}$ and $\beta=3-2\sqrt{2}$, this formula is called the Binet formula of the sequence. One may refer to \cite{PA}, where the author proved interesting properties of this sequence of numbers. Many of their arithmetic properties were studied in  \cite{PA1,DU,SA}.
 
There are some generalizations of the Balancing and Balancing-Lucas numbers.
In 2014, Tecan and others defined the $k$-balancing numbers such that, for any positive number $k$, the $k$- balancing number defined by   \(B_{k,n+1}=6kB_{k,n}-B_{k,n-1}\), 
where $n\geq 1$ and $k\geq1$, with the initials  $B_{k,0}=0$ and $B_{k,1}=1,$ see \cite{TE,RA2}.
The Gaussian balancing sequence was defined in 2018 by Tasci  in   \cite{GA} to be  the sequence of complex numbers $GB_n$
defined by the initial values $GB_0=i$ and $GB_1=1$, and the recurrence relation
$GB_{n+1}=6GB_n-GB_{n-1}$, for all  $n\geq 1$.
The first few values of are $i$,~$1$,~$6-i$,~$35-i$,...

In this study we introduce    generalizations of Balancing and Balancing-Lucas numbers and examine the properties of the new  sequences. The paper is structured as follows: in section 1 we introduce the new sequences and study their generating functions and Binet formulas. In section 2 we provide some basic identities of the generalized Balancing and balancing-Lucas numbers. Also, we find a matrix representation for both sequences and investigate some divisibility properties.
\section{Definition, Generating Functions and Binet Formula}
In this section we define the generalized Balancing  and balancing- Lucas numbers, and we derive their generating functions and Binet formulas.
\begin{defn} [Generalized Balancing Numbers] \label{d2}For $n\geq 2$ and $k\geq1 $ we define \[B_{k,n}=3kB_{k,n-1}+(1-k)B_{k,n-2}\] 
where $B_{k,0}=0$ and $B_{k,1}=1.$
\end{defn} Clearly, \begin{align*} B_{1,n}&= 3B_{1,n-1}\\
                    &= 3\cdot3\cdot B_{1,n-2}\\
                      & \vdots\\
                      &= 3^n
 \end{align*} which leads to the following remark.
\begin{rem} $B_{1,n}=3^n$ for $n\geq 2$.
\end{rem}
\begin{rem} $B_{2,n}$ is the  original sequence  of Balancing numbers. That is, $B_{2,n}=6B_{2,n-1}-B_{2,n-2}$.
\end{rem} 
\begin{defn} [Generalized Balancing-Lucas Numbers]\label{d1} For $n\geq 2$ and $k\geq1 $ we define \[C_{k,n}=3kC_{k,n-1}+(1-k)C_{k,n-2}\] 
where $C_{k,0}=1$ and $C_{k,1}=3.$
\end{defn}
\begin{rem} $C_{1,n}=3^n$ for $n\geq 2$.
\end{rem}
\begin{rem} $C_{2,n}$ is the  original sequence  of Balancing-Lucas numbers. That is, $C_{2,n}=6C_{2,n-1}-C_{2,n-2}$.
\end{rem}

\begin{exa} Consider the values of $B_{k,n}, C_{k,n}$ for different values of $k$
\begin{center}$\begin{array}{|c|c|c|c|c|c|c|c|c|c|c|c|c|c|c|c|c|c|c|c|} \hline
n & B_{1,n} & B_{2,n} & B_{3,n} & B_{4,n}&B_{k,n}&C_{1,n} & C_{2,n} & C_{3,n} & C_{4,n} & C_{k,n} \\\hline
0 &  0& 0&0 &    0&0&1&1&1&1&1\ \\\hline
1 & 1 & 1 & 1 &1&1 &3&3&3&3&3      \\\hline
2 & 9 & 6 & 9&  12&3k&9 &17 &25&33 &8k+1\\\hline
3 & 27 & 35& 79&  141&9k^2-k+1&27&99 &219&387&24k^2+3  \\\hline
4 & 81 & 1189&693 &1656 &3k(9k^2-2k+2)&81&577 &1921 &4545&27k^3-8k^2+16k+1  \\\hline
5 & 243 &6930 & 6079&   19449 &*&243&3363&16851&53379&* \\\hline
\end{array}$
\end{center}
\end{exa}

\begin{thm}[Binet Formulas]\label{thm:1} Let $k\geq 1$ and $n\geq 2$. Then we have \begin{enumerate}
\item $B_{k,n}=\frac{\alpha^n-\beta^n}{\alpha-\beta}$. 
\item  $C_{k,n}= \frac{\beta-3}{\beta-\alpha} \alpha^n+ \frac{3-\alpha}{\beta-\alpha}\beta^n$.
\end{enumerate}
 
Where $\alpha=\frac{3k+\sqrt{9k^2-4k+4}}{2}, \beta=\frac{3k-\sqrt{9k^2-4k+4}}{2}$.
\end{thm}
\begin{proof}

We prove both identities by mathematical induction. For $B_{k,n}$:\begin{itemize}
\item if $m=2$ ; $\frac{\alpha^2-\beta^2}{\alpha-\beta}=\alpha+\beta= \frac{3k+ \sqrt{9{k}^{2}-4k+4}}{2}+\frac{3k-\sqrt{9{k}^{2}-4k+4}}{2}=\frac{6k}{2}=3k=B_{k,2}$,
\item  Assume that $B_{k,n}=\frac{\alpha^n-\beta^n}{\alpha-\beta}$, 
\item  Consider 
\begin{align*}
 B_{k,n+1}&=3kB_{k,n}+(1-n)B_{k,n-1}\\&=3k \frac{\alpha^n -\beta^n}{\alpha-\beta}+(1-k)\frac{\alpha^{n-1}-\beta^{n-1}}{\alpha-\beta} \\
&=\frac{3k\alpha^{n}-3k\beta^{n}+\alpha^{n-1}\beta^{n-1}-k\alpha^{n-1}+k\beta^{n-1}}{\alpha-\beta}\\ &=\frac{\alpha^{n-1}(3k\alpha+1-k)-\beta^{n-1}(3k\beta+1-k)}{\alpha-\beta}\\ &=\frac{\alpha^{n-1}\alpha^{2}-\beta^{n-1}\beta^{2}}{\alpha-\beta}\\
&=\frac{\alpha^{n+1}-\beta^{n+1}}{\alpha-\beta}.
\end{align*}
\end{itemize} 
 Proving the formula for $C_{k,n}$ also can be done similarly by induction.
\end{proof} 
\begin{rem} For  $n \geq2$, we have  ${\alpha\beta}=k-1$.
\end{rem}

Next we extend Definition \ref{d2} to negative indices as follows:  \[B_{k,-n}=\frac{\alpha^{-n}-\beta^{-n}}{\alpha-\beta}=(\frac{1}{\alpha^n \beta^n})(-B_{k,n})\]
but, \[\frac{1}{\alpha^n\beta^n}=(\frac{4}{9k^2- 9k^2+4k-4})^{n} =(\frac{1}{k-1})^n \]
thus, $B_{k,-n}=-(k-1)^{-n} B_{k,n}$
 \begin{defn}\label{def:3}  $B_{k,-n}= -(1-k)^{-n} B_{k,n}$.
\end{defn}

\begin{thm}[Generating Functions]\label{thm:2} The generating functions of the sequences $B_{k,n}$
and $C_{k,n}$ respectively are \begin{enumerate}
\item $B(x)=\frac{x}{1-3kx+(k-1)x^2}$
\item $C(x)=\frac{1+3x(1+k)}{1-3kx+(k-1)x^2}$
\end{enumerate}
\end{thm}
\begin{proof}Let $B(x)$ and $C(x)$ represents the generating functions of $B_{k,n}$ and $C_{k,n}$ respectively.
Note,
 \begin{align*} B(x)&= \sum_{n=0}^{\infty} B_{k,n}x^n\\
                    &= B_{k,0}+B_{k,1}x+\sum_{n=2}^{\infty} (3kB_{k,n-1}+(1-k)B_{k,n-2})x^n\\
                      &= x+3kx\sum_{n=2}^{\infty} B_{k,n-1}x^{n-1}+(1-k)x^2\sum_{n=2}^{\infty}B_{k,n-2}x^{n-2}\\
                      &= x+3kx\sum_{n=0}^{\infty} B_{k,n}x^{n}+(1-k)x^2\sum_{n=0}^{\infty}B_{k,n}x^{n}\\
                      &= x+3kx B(x)+(1-k)x^2B(x)\\
 \end{align*}
 thus $B(x)(1-3kx+(k-1)x^2)=x\Rightarrow B(x)=\frac{x}{(1-3kx+(k-1)x^2)}$.
 Next for $C_{k,n},$  Note\begin{align*} C(x)&=\sum_{n=0}^{\infty} C_{k,n}x^n\\ &=C_{k,0}+C_{k,1}x+\sum_{n=2}^{\infty}(3k C_{k,n-1}+(1-k)C_{k,n-2})x^n \\ &=1+3x+3kx \sum_{n=2}^{\infty}C_{k,n-1}x^{n-1}+(1-k)x^2 \sum_{n=2}^{\infty}C_{k,n-2}x^{n-2} \\ &=1+3x+3kx \sum_{n=1}^{\infty}C_{k,n}x^{n}+(1-k)x^2 \sum_{n=0}^{\infty}C_{k,n}x^{n} \\ &=1+3x+3kx (-1+\sum_{n=0}^{\infty}C_{k,n}x^{n})+(1-k)x^2 \sum_{n=0}^{\infty}C_{k,n}x^{n} \\ &=1+3x-3kx+3kx \sum_{n=0}^{\infty}C_{k,n}x^{n}+(1-k)x^2 \sum_{n=0}^{\infty}C_{k,n}x^{n} \\ &=1+3x-3kx+3kx C(x)+(1-k)x^2 C(x)=C(x)\\
 \end{align*}
 then,\\
     \[C(x)=\frac{1+3x-3kx}{1-3kx+(k-1)x^2}\]
     \[=\frac{1+3x(1-k)}{1-3kx+(k-1)x^2}\]
\end{proof}

\section{Main Results}
\subsection{Basic Identities} In this subsection we provide several basic identities like Catalan identity, d' Ocagnes identity, and Vajda's identity. Also  we give the sum of terms of the first $n$ terms of the sequences $B_{k,m}$ and $C_{k,m}$. 
\begin{thm}[Catalan's Identity 1]\label{thm:5}~Let $n\geq 2$ be an integer. Then \[B_{k,n+r}B_{k,n-r}-B_{k,n}^2=-(k-1)^{n-r} (B_{k,r})^2\]
\end{thm}\begin{proof} By the Binet formula we have:
\begin{align*}B_{k,n+r}B_{k,n-r}-B_{k,n}^2=&\frac{\alpha^{n+r}-\beta^{n+r}}{\alpha-\beta}\frac{\alpha^{n-r}-\beta^{n-r}}{\alpha-\beta}-\left(\frac{\alpha^{n}-\beta^{n}}{\alpha-\beta}\right)^2\\
&= \frac{\alpha^{2n}-\alpha^{n+r}\beta^{n-r}-\alpha^{n-r}\beta^{n+r}+\beta^{2n}}{(\alpha-\beta)^2}\\&-\frac{\alpha^{2n}-2\alpha^n\beta^n+\beta^{2n}}{(\alpha-\beta)^2}\\
&= \frac{2\alpha^n\beta^n-\alpha^{n+r}\beta^{n-r}-\alpha^{n-r}\beta^{n+r}}{(\alpha-\beta)^2}\\
&=(\alpha\beta)^{n-r} \frac{2\alpha^r\beta^r-\alpha^{2r}-\beta^{2r}}{(\alpha-\beta)^2}\\
&=-(\alpha\beta)^{n-r} \frac{\alpha^{2r}-2\alpha^r\beta^r+\beta^{2r}}{(\alpha-\beta)^2}\\
&=-(\alpha\beta)^{n-r} (B_{k,r})^2 \end{align*}
\end{proof}
\begin{thm}[Catalan's Identity 2]\label{thm:55}~Let $n\geq 2$ be an integer. Then \[C_{k,n+r}C_{k,n-r}-C_{k,n}^2=8(k-1)^{n-r+1} B^2_{k,r}\]
\end{thm}
\begin{proof}Let $A= \frac{\beta-3}{\beta-\alpha}$ and $B= \frac{3-\alpha}{\beta-\alpha}$. Then 
\begin{align*}C_{k,n+r}C_{k,n-r}-C_{k,n}^2&= (A\alpha^{n+r}+B\beta^{n+r})(A\alpha^{n-r}+B\beta^{n-r})\\
& \cdot (A\alpha^{n}+B\beta^{n})^2\\
&= AB(\alpha^{n+r}\beta^{n-r}+\alpha^{n-r}\beta^{n+2}-2\alpha^n\beta^n)\\
&= AB (\alpha \beta)^{n-r}(\alpha^{2r}+\beta^{2r}-2\alpha^r\beta^r)\\
&= (\beta-3)(3-\alpha) (\alpha \beta)^{n-r} \frac{ (\alpha^{r}-\beta^{r})^2}{\alpha-\beta^2}\\
&=  8(k-1)^{n-r+1} B_{k,n}^2&(\beta-3)(3-\alpha)=8(k-1) \\
\end{align*}
\end{proof}
\begin{thm}[ d' Ocagne's identity 1] For any two integers $m$ and $n$ with $m\geq n$\[B_{k,m}B_{k,n+1}-B_{k,n}B_{k,m+1}=(k-1)^nB_{k,m-n}\]
\end{thm}\begin{proof}Note

\begin{align*}B_{k,m}B_{k,n+1}-B_{k,n}B_{k,m+1}&=\frac{(\alpha^m-\beta^m)(\alpha^{n+1}-\beta^{n+1})-(\alpha^n-\beta^n)(\alpha^{m+1}-\beta^{m+1})}{(\alpha-\beta)^2}\\
&= \frac{-\alpha^m\beta^{n+1}-\alpha^{n+1}\beta^{m}+\alpha^n\beta^{m+1}+\alpha^{m+1}\beta^{n}}{(\alpha-\beta)^2}\\
&= \frac{(\alpha\beta)^n}{(\alpha-\beta)^2}(-\beta\alpha^{m-n}-\alpha\beta^{m-n}+\beta^{m-n+1}+\alpha^{m-n+1})\\ &= \frac{(\alpha\beta)^n}{(\alpha-\beta)^2}((\alpha-\beta)\alpha^{m-n}-(\alpha-\beta)\beta^{m-n})\\
&= (k-1)^nB_{k,m-n} \end{align*}\end{proof}
\begin{thm}[d' Ocagne's identity $2$] For any two integers $m$ and $n$ with $m\geq n$ we have \[C_{k,m}C_{k,n+1}-C_{k,n}C_{k,m+1}=-8(k-1)^{n+1}B_{k,m-n}\]
\end{thm}
\begin{proof}Let $A= \frac{\beta-3}{\beta-\alpha}$ and $B= \frac{3-\alpha}{\beta-\alpha}$. Then 
\begin{align*}C_{k,m}C_{k,n+1}-C_{k,n}C_{k,m+1}&= (A\alpha^m+B\beta^m)(A\alpha^{n+1}+B\beta^{n+1})\\
&-(A\alpha^n+B\beta^n)(A\alpha^{m+1}+B\beta^{m+1})\\&=AB(\alpha^m\beta^{n+1}+\alpha^{n+1}\beta^{m}-\alpha^n\beta^{m+1}-\alpha^{m+1}\beta^{n})\\
&= AB(\alpha\beta)^{n} (\alpha^{m-n}\beta+\alpha^{}\beta^{m-n}-\beta^{m-n+1}-\alpha^{m-n+1})\\
&= AB(\alpha\beta)^{n} (\alpha^{m-n}(\beta-\alpha)+\beta^{m-n}(\alpha-\beta))\\
&= AB(\beta-\alpha)(\alpha\beta)^{n} (\alpha^{m-n}-\beta^{m-n})\\
&= -(\beta-3)(3-\alpha)(\alpha\beta)^{n} \frac{\alpha^{m-n}-\beta^{m-n}}{\alpha-\beta}\\
&= -8(k-1)^{n+1}B_{k,m-n}
\end{align*}\end{proof}

\begin{thm}[Vajda's identity] Let  $n,m,i,j,k,\ell$ be positive integers with $m>n+\ell$. Then 
\begin{enumerate}
\item (Formulation 1) $B_{k,n+i}B_{k,n+j}-B_{k,n}B_{k,n+i+j}=(k-1)^nB_{k,i}B_{k,j}$
\item (Formulation 2) $B_{k,n+\ell}B_{k,m-\ell}-B_{k,n}B_{k,m}=(k-1)^{n}B_{k,m-n-\ell}B_{k,\ell}$
\end{enumerate}
\end{thm}\begin{proof}We prove both identities using the Binet formula:\begin{enumerate}
\item \begin{align*}B_{k,n+i}B_{k,n+j}-B_{k,n}B_{k,n+i+j}&=\frac{(\alpha^{n+i}-\beta^{n+i})(\alpha^{n+j}-\beta^{n+j})}{(\alpha-\beta)^2}\\
&-\frac{(\alpha^{n}-\beta^{n})(\alpha^{n+i+j}-\beta^{n+i+j})}{(\alpha-\beta)^2}\\
&= \frac{\alpha^n\beta^{n+i+j}-\alpha^{n+i}\beta^{n+j}-\alpha^{n+j}\beta^{n+i}+\beta^n\alpha^{n+i+j}}{(\alpha-\beta)^2}\\
&= (\alpha\beta)^n\frac{\beta^{i+j}-\alpha^i\beta^j-\alpha^j\beta^i+\alpha^{i+j}}{(\alpha-\beta)^2}\\
&= (k-1)^nB_{k,i}B_{k,j}\\\end{align*}  
\item \begin{align*}B_{k,n+\ell}B_{k,m-\ell}-B_{k,n}B_{k,m}&=\frac{\alpha^n\beta^m+\alpha^m\beta^n-\alpha^{n+\ell}\beta^{m-\ell}-\alpha^{m-\ell}\beta^{n+\ell}}{(\alpha-\beta)^2}\\
&= \frac{\alpha^n\beta^m+\alpha^m\beta^n-\alpha^{n+\ell}\beta^{m-\ell}-\alpha^{m-\ell}\beta^{n+\ell}}{(\alpha-\beta)^2}\\
&=(\alpha\beta)^n\frac{\beta^{m-n}+\alpha^{m-n}-\alpha^{\ell}\beta^{m-n-\ell}-\alpha^{m-n-\ell}\beta^{\ell}}{(\alpha-\beta)^2}\\
&=(1-k)^n B_{k,m-n-\ell}B_{k,\ell}.
\end{align*}
\end{enumerate}
\end{proof}
\begin{thm}\label{thm:sum1}
Let $n \geq 0$ be an integer. The sums of first n-terms of generalized balancing  and balancing-Lucas numbers are\ given by \begin{enumerate}
\item $\sum_{i=0}^{n} B_{k,n}=\frac{-(2k+1)B_{k,n}+(k-1)B_{k,n-1}+1}{-2k}$ 
\item $      \sum_{i=2}^{n}C_{k,n}=\frac{-(2k+1)C_{k,n}+(k-1)C_{k,n-1}+4(1-k)}{-2k}$ 
\end{enumerate}
   
\end{thm}
\begin{proof}
For the first sum, note
\begin{align*}
S&=\sum_{i=2}^n B_{k,i}\\
&= \sum_{i=2}^n (3kB_{k,i-1}+(1-k)B_{k,n-2})\\
&= 3k\sum_{i=2}^{n}B_{k,i-1}+(1-k)\sum_{i=2}^{n} B_{k,i-2}\\&= 3k\sum_{i=1}^{n-1}B_{k,i}+(1-k)\sum_{i=2}^{n-2}B_{k,i}\\&= 3k(B_{k,1}-B_{k,n}+S)+(1-k)(B_{k,0}+B_{k,1}-B_{k,n}-B_{k,n-1}+S)\\
-2kS&= 3k(1-B_{k,n})+(1-k)(1-B_{k,n}-B_{k,n-1})
\end{align*}Thus, $\sum_{i=2}^n B_{k,n}= \frac{(2k+1)-(2k+1)B_{k,n}+(k-1)B_{k,n-1}}{-2k}$. Also, \[\sum_{i=0}^nB_{k,n}=1+\frac{(2k+1)-(2k+1)B_{k,n}+(k-1)B_{k,n-1}}{-2k}=\frac{-(2k+1)B_{k,n}+(k-1)B_{k,n-1}+1}{-2k}\]
Next, for the first sum, note
\begin{align*}
S&=\sum_{i=2}^n C_{k,i}\\
&= \sum_{i=2}^n (3kC_{k,i-1}+(1-k)C_{k,n-2})\\
&= 3k\sum_{i=2}^{n}C_{k,i-1}+(1-k)\sum_{i=2}^{n} C_{k,i-2}\\&= 3k\sum_{i=1}^{n-1}C_{k,i}+(1-k)\sum_{i=2}^{n-2}C_{k,i}\\&= 3k(C_{k,1}-C_{k,n}+S)+(1-k)(C_{k,0}+C_{k,1}-C_{k,n}-C_{k,n-1}+S)\\
-2kS&= 3k(3-C_{k,n})+(1-k)(1+3-C_{k,n}-C_{k,n-1})
\end{align*}Thus, $\sum_{i=2}^n C_{k,n}= \frac{(4+5k)-(2k+1)C_{k,n}+(k-1)C_{k,n-1}}{-2k}$. Also, \[\sum_{i=0}^nC_{k,n}=1+3+\frac{(5k+4)-(2k+1)C_{k,n}+(k-1)C_{k,n-1}}{-2k}=\frac{-(2k+1)C_{k,n}+(k-1)C_{k,n-1}+4(1-k)}{-2k}\]
\end{proof}

\subsection{Matrix Representation} As we know most integer sequences have matrix representation, for example the balancing numbers have the matrix $Q=\begin{pmatrix}6 & -1 \\
1 & 0 \\
\end{pmatrix}$  and $Q^n=\begin{pmatrix}B_{n+1} & -B_n \\
B_n & -B_{n-1} \\
\end{pmatrix}$. Also, the matrix $S=\begin{pmatrix}3 & 8 \\
1 & 3 \\
\end{pmatrix}$ for which $S^n=\begin{pmatrix}C_n & 8B_n \\
B_n & C_n \\
\end{pmatrix}$. the last matrices were used to prove many properties for the balancing and balancing-Lucas numbers, see \cite{RA1}. 
\begin{thm}\label{thm:3} Let $A=\begin{bmatrix}3k & 1-k \\
1 & 0 \\
\end{bmatrix}$. Then  for $n\geq2$ we have \[A^n=\begin{bmatrix}B_{k,n+1} & (1-k)B_{k,n} \\
B_{k,n} & (1-k)B_{k,n-1} \\
\end{bmatrix}\]
\end{thm}
\begin{proof}
We will prove this property by  mathematical induction.\\
For $m=2$,\begin{align*}
    A^2&=\begin{bmatrix}3k & 1-k \\
1 & 0 \\
\end{bmatrix}\begin{bmatrix}3k & 1-k \\
1 & 0 \\
\end{bmatrix}\\ &=\begin{bmatrix}9{k}^2 & 3k(1-k) \\
3k & 1-k \\\end{bmatrix}\\
&=\begin{bmatrix}B_{k,3} & (1-k)B_{k,2} \\
B_{k,2} & (1-k)B_{k,1} \\
\end{bmatrix}
\end{align*}
Now, assume that the property is true  $m=n$,
now consider  $m=n+1$,\begin{align*} A^{n+1}&=A^{n}A\\
&=\begin{bmatrix} B_{k,n+1} & (1-k)B_{k,n} \\ B_{k,n} & (1-k)B_{k,n-1}\\
\end{bmatrix}  \begin{bmatrix} 3k & 1-k \\ 1 & 0 \\ 
\end{bmatrix}\\ &=\begin{bmatrix} 3kB_{k,n+1}+(1-k)B_{k,n} & (1-k)B_{k,n+1} \\ 3kB_{k,n}+(1-k)B_{k,n-1} & (1-k)B_{k,n}\\ 
\end{bmatrix}\\ &=\begin{bmatrix} B_{k,n+2} & (1-k)B_{k,n+1} \\ B_{k,n+1} & (1-k)B_{k,n}\end{bmatrix}
\end{align*}as desired.
\end{proof}
\begin{cor} [Cassini Identity 1] For $n \geq2$, we have $B_{k,n}^2-B_{k,n-1}B_{k,n+1}=(k-1)^{n-1}$.
\end{cor}
\begin{proof}
\begin{align*} \det A &= \begin{vmatrix} 3k & 1-k \\ 1 & 0 \end{vmatrix}\\ &=k-1
\end{align*}
Then ,
\begin{align*} \det A &= \begin{vmatrix} B_{k,n+1} & (1-k)B_{k,n} \\ B_{k,n} & (1-k)B_{k,n-1} \end{vmatrix}\\ (k-1)^n &=(1-k)B_{k,n-1} B_{k,n+1}-(1-k){B_{k,n}}^2\\
 (k-1)^{n-1} &={B_{k,n}}^{2}-B_{k,n-1} B_{k,n+1}
\end{align*}
\end{proof}
\begin{thm}\label{thm:4} Let $A=\begin{bmatrix}3k & 1-k \\
1 & 0 \\
\end{bmatrix}, R=\begin{bmatrix}3 & 1-k \\
1 & 3(1-k) \\
\end{bmatrix}$. Then for $n\geq2$ we have  \[R_n=RA^n=\begin{bmatrix}C_{k,n+1} & (1-k)C_{k,n} \\
C_{k,n} & (1-k)C_{k,n-1} \\
\end{bmatrix}\].
\end{thm}
\begin{proof} We prove this property by mathematical induction.
For $m=2$,\begin{align*} R_2&=RA^2 \\ &=\begin{bmatrix} 3 & (1-k) \\ 1 & 3(1-k)\\
\end{bmatrix} \begin{bmatrix} 9k^2+1-k & 3k(1-k) \\ 3k & 1-k \\
\end{bmatrix} \\ &=\begin{bmatrix} 24k^2+3 & (1-k)(8k+1) \\ 8k+1 & 3(1-k) \\ 
\end{bmatrix} \\ &=\begin{bmatrix} C_{k,3} & (1-k)C_{k,2} \\ C_{k,2} & (1-k)C_{k,1} \\\end{bmatrix}
\end{align*}\\
\text{ Assume that the property is  true for} $m=n$, 
now consider  $m=n+1$;
\begin{align*} RA^{n+1}&=R_{n}A \\
    &=\begin{bmatrix}C_{k,n+1} & (1-k)C_{k,n} \\
C_{k,n} & (1-k)C_{k,n-1} \\
\end{bmatrix}\begin{bmatrix}3k & 1-k \\ 1 & 0 \\ 
\end{bmatrix}\\
&=\begin{bmatrix} 3kC_{k,n+1}+(1-k)C_{k,n} & (1-k)C_{k,n+1} \\ 3kC_{k,n}+(1-k)C_{k,n-1} & (1-k)C_{k,n}\\
\end{bmatrix}\\
&=\begin{bmatrix} C_{k,n+2} & (1-k)C_{k,n+1} \\ C_{k,n+1} & (1-k)C_{k,n}\\
\end{bmatrix}
\end{align*}
\end{proof}

\begin{rem} $AR=RA$
\end{rem}
\begin{cor} [Cassini Identity 2] For $n \geq2$, we have \[C_{k,n}^2-C_{k,n-1}C_{k,n+1}=-8(k-1)^{n}\]
\end{cor}
\begin{proof}
\begin{align*} \det R_{n}&=\det R \det A^n \\ &=8(1-k)(k-1)^n \\ &=-8(k-1)^{n+1} \\ &= \begin{vmatrix} C_{k,n+1} & (1-n)C_{k,n} \\ C_{k,n} & (1-k)C_{k,n-1} \end{vmatrix}\\
-8(k-1)^{n+1} &=(1-k)C_{k,n-1} C_{k,n+1} - (1-k)C_{{k,n}}^2 \\
-8(k-1)^{n} &={C_{k,n}}^2 - C_{k,n-1} C_{k,n+1} \end{align*}\\
\end{proof}
\begin{cor}\label{cor:5}$C_{k,n}=B_{n+1}+3(1-k)B_n$.
\end{cor}
\begin{proof} Consider the matrix $R_n=RA^n=\begin{bmatrix}C_{k,n+1} & (1-k)C_{k,n} \\
C_{k,n} & (1-k)C_{k,n-1} \\
\end{bmatrix}$ from Theorem \ref{thm:4}, we have \\$R_n=R\begin{bmatrix}B_{k,n+1} & (1-k)B_{k,n} \\
B_{k,n} & (1-k)B_{k,n-1} \\
\end{bmatrix}$ from Theorem \ref{thm:3}. Thus, \[\begin{bmatrix}C_{k,n+1} & (1-k)C_{k,n} \\
C_{k,n} & (1-k)C_{k,n-1} \\
\end{bmatrix}=\begin{bmatrix}3 & 1-k \\
1 & 3(1-k) \\
\end{bmatrix}\begin{bmatrix}B_{k,n+1} & (1-k)B_{k,n} \\
B_{k,n} & (1-k)B_{k,n-1} \\
\end{bmatrix}\]. So by equating the corresponding entries we get the desired result.
\end{proof}
\begin{lem} \label{lem:1} For all integers  $n \geq2$ , we have \[\alpha^n+\beta^n=B_{k,n+1}-(k-1)B_{k,n-1}\]
\end{lem}
\begin{proof}
By Bient formula : \[B_{k,n+1}=\frac{\alpha^{n+1}-\beta^{n+1}}{\alpha-\beta}=\alpha^n +\alpha^{n-1} \beta +\alpha^{n-2} \beta^2 +.......+\alpha \beta^{n-1}+\beta^n\]
Also, \begin{align*}
  (k-1)B_{k,n-1}&=(k-1)(\frac{\alpha^{n-1}-\beta^{n-1}}{\alpha-\beta})\\ &=(k-1)(\alpha^{n-2}+\alpha^{n-3}\beta+......+\alpha \beta^{n-3}+\beta^{n-2})\\ &= \alpha^{n-1} \beta+\alpha^{n-2}\beta^2+......+\alpha^2 \beta^{n-2}+\alpha \beta^{n-1}
\end{align*}
Then we have \begin{align*}
    B_{k,n+1}-(k-1)B_{k,n-1} &=( \alpha^n +\alpha^{n-1}\beta +\cdots+\beta^n )- (\alpha^{n-1}\beta +\alpha^{n-2}\beta^2+\cdots+\alpha \beta^{n-1})\\ &=\alpha^n + \beta^n .
\end{align*}
\end{proof}
\begin{lem}\label{lem:2} For any two positive integers $m$ and $n$ with $n \geq2$, we have \[B_{k,m+n}=B_{k,m}B_{k,n+1}+(1-k)B_{k,m-1}B_{k,n}\]
\end{lem}
\begin{proof}
\begin{align*}A^{n+m} =A^nA^m&=\begin{bmatrix}B_{k,n+m+1} & (1-k)B_{k,n+m} \\
B_{k,n+m} & (1-k)B_{k,n+m-1}
\end{bmatrix} \\ &={\begin{bmatrix}B_{k,m+1} & (1-k)B_{k,m}\\
B_{k,m} & (1-k)B_{k,m-1} \\
\end{bmatrix}}{\begin{bmatrix}B_{k,n+1} & (1-k)B_{k,n}\\
B_{k,n} & (1-k)B_{k,n-1} \\
\end{bmatrix}}\\ &=\begin{bmatrix} *&* \\
B_{k,m}B_{k,n+1}+(1-k)B_{k,m-1}B_{k,n} &*  \\
\end{bmatrix}
\end{align*}
Thus $B_{k,m+n}=B_{k,m}B_{k,n+1}+(1-k)B_{k,m-1}B_{k,n}$ as desired.
\end{proof}
\begin{cor}Let $n\geq 1$ be a natural integer. Then \begin{enumerate}
\item $B_{k,2n}=B_{k,n}(B_{k,n+1}+(1-k)B_{k,n-1})$
\item$B_{k,2n-1}=B^2_{k,n}+(1-k)B^2_{k,n-1}$
\end{enumerate}
\end{cor}\begin{proof}
By Lemma \ref{lem:2}\begin{enumerate}
\item $B_{k,2n}=B_{k,n}B_{k,n+1}+(1-k)B_{k,n-1}B_{k,n}=B_{k,n}(B_{k,n+1}+(1-k)B_{k,n-1})$
\item $B_{k,2n-1}=B_{k,n}B_{k,n}+(1-k)B_{k,n-1}B_{k,n-1}=B^2_{k,n}+(1-k)B^2_{k,n-1}$
 
\end{enumerate}\end{proof}
\subsection{Divisibility Properties}

\begin{lem}\label{lem:3} Let $m,n\geq 2$ be two integers. If $m|n$, then $B_{k,m}|B_{k,n}$.
\end{lem}
\begin{proof}If $m|n$, then $n=mt$ for some integer $t$. We can prove the property by induction on $t$.\begin{enumerate}
\item If $t=1$, then it is trivial that $B_{k,m}|B_{k,m}.$
\item Assume that $B_{k,m}|B_{k,mj}$ for all $j =1,2,\cdots,t.$
\item Consider $B_{k,m(t+1)}=B_{k,mt+m}=B_{k,mt+1}B_{k,m}+(1-k)B_{k,mt}B_{k,m-1},$ by Lemma \ref{lem:2}. Now, from induction step 2 we may write $B_{k,mt}=B_{k,m}T$  for some integer $T.$
Thus  $B_{k,m(t+1)}=B_{k,m}(B_{k,mt+1}+(1-k)B_{k,m-1}T)$ which leads to $B_{k,m}|B_{k,m(t+1)}$.
\end{enumerate}
\end{proof}
\begin{lem}\label{lem:4} If $k \not\equiv 1 \mod 3$, then $\gcd(1-k,B_{k,n})=1,$ for $n \geq 1.$
\end{lem}
\begin{proof} We can prove this property by  induction on $n$.\begin{enumerate}
\item Since $B_{k,1}=1$ we have $\gcd(1-k,B_{k,1})=1$.
\item Assume that $\gcd(1-k,B_{k,m})=1$ for all integers $m=1,\cdots,n-1.$
\item Let $d=\gcd(1-k,B_{k,n})$. Then $d$ divides both integer and any linear combination of them, that is $d| (B_{k,n}- (1-k)B_{k,n-2})=3kB_{k,n-1}$. Now, since $d|(1-k)$ we know that $d\neq3$ and $d\nmid k$. Thus $d|B_{k,n-1}$, so we have $d$ divides both $1-k$ and $B_{k,n-1}$ but from induction step 2 we know that $\gcd(1-k, B_{k,n-1})=1$. Hence $d=1.$ 
\end{enumerate}
\end{proof}
We need the next remark in the proof of Theorem \ref{thm:6}. In fact it is exercise 15 in section 3.3 in \cite{RO}.
\begin{rem}\label{rem:3}For any three integers $a,b$ and $c$ we have \[\gcd(a,bc)=\gcd(a,b)\gcd(a,c)\]
\end{rem}
\begin{thm}\label{thm:6} Assume that  $n\geq 1$ and $k \geq 1$ are two integers such that   $k \not\equiv\ 1 \mod 3$.  Then $\gcd
 (B_{k,n},B_{k,n+1})=1$.
\end{thm}
\begin{proof}
We will prove the theorem by mathematical induction.\begin{enumerate}
\item When $n=1, \gcd(B_{k,1},B_{k,2})=\gcd(1,B_{k,2})=1$.
\item Assume that $\gcd(B_{k,m}, B_{k,m+1})=1$ for all $m=1,2,\cdots, n.$
\item Now, let  $d=\gcd(B_{k,n+1},B_{k,n+2})$. We have $d|B_{k,n+1}$ and $d|B_{k,n+2}$ so it divides any linear combination of them that is \[d|(B_{k,n+2}-3kB_{k,n+1})=(1-k)B_{k,n}\] Thus,  $d|(\gcd(B_{k,n+1},(1-k)B_{k,n}))$. Now, from Remark \ref{rem:3} we have \[\gcd(B_{k,n+1},(1-k)B_{k,n})=\gcd(B_{k,n+1},B_{k,n})\gcd(B_{k,n+1},1-k)=1\]
by Lemma \ref{lem:4} and the second step of induction. Which leads to the fact that $d=1.$\end{enumerate} 
\end{proof}
\begin{lem}\label{lem:6} If $k \not\equiv 1 \mod 3$, then $\gcd(1-k,C_{k,n})=1,$ for $n \geq 1.$
\end{lem}
\begin{proof}The proof of this lemma is similar to the proof of Lemma \ref{lem:4}
\end{proof}
\begin{thm}\label{thm:7} Assume that  $n\geq 1$ and $k \geq 1$ are two integers such that   $k \not\equiv\ 1 \mod 3$.  Then $\gcd
 (C_{k,n},C_{k,n+1})=1$.
\end{thm}
\begin{proof} We omit the proof of this theorem since it is similar to the proof of Theorem \ref{thm:6}.
\end{proof}
\begin{thm}\label{thm:8} Assume that  $n\geq 0$ and $k \geq 1$ are two integers with  $k \not\equiv\ 1 \mod 3$. Then  $\gcd
 (B_{k,n},C_{k,n})=1$.
\end{thm}
\begin{proof} From 
Corollary \ref{cor:5}, $C_{k,n}=B_{n+1}+3(1-k)B_{k,n}.$ Now, $\gcd
 (B_{k,n},C_{k,n})=\gcd(B_{k,n},B_{n+1}+3(1-k)B_{k,n})=\gcd(B_{k,n},B_{k,n+1})=1$.
\end{proof}
\begin{thm}\label{thm:9} Let $m,n\geq 1$ be two integers and $k\geq 1$ be an integer such that $k\not\equiv 1 \mod 3$.  Then $\gcd(B_{k,m},B_{k,n})=B_{k,\gcd(m,n)}$.
\end{thm}
\begin{proof} We can prove this theorem by mathematical induction. Without loss of generality we may assume that $m\leq n.$ The result is trivial when $n=1,$ so assume that $\gcd(B_{k,m},B_{k,j})=B_{k,\gcd(m,j)}$ for all $j <n.$ Next use the division algorithm to write $n=mq+r,$ where $0\leq r<m.$ Thus, $B_{k,n}=B_{k,mq+r}= B_{k,mq+1}B_{k,r}+(1-k)B_{k,mq}B_{k,r-1}$, by Lemma \ref{lem:2}. Also, by Lemma \ref{lem:3} we have $B_{k,m}|B_{k,mq}$, so we have \[B_{k,n}= B_{k,mq+1}B_{k,r}+(1-k)xB_{k,m}B_{k,r-1}\] 
for some integer $x$. Note, \begin{align*} 
\gcd(B_{k,n}, B_{k,m})&=\gcd(B_{k,mq+r},B_{k,m})\\
&=\gcd(B_{k,mq+1}B_{k,r}+(1-k)xB_{k,m}B_{k,r-1},B_{k,m})\\&=\gcd(B_{k,mq+1}B_{k,r},B_{k,m})\end{align*}
Suppose that $d=\gcd(B_{k,n},B_{k,m})=\gcd(B_{k,mq+1}B_{k,r},B_{k,m}).$ Then $d|B_{k,m}$ and $d|B_{k,mq+1}B_{k,r}$, by Theorem \ref{thm:6} we have $d\nmid B_{k,mq+1}$ thus $d|B_{k,r}$ which gives us that $d|\gcd(B_{k,m},B_{k,r}).$ Assume that $d^{\prime}=\gcd(B_{k,m},B_{k,r}),$ it is clear that $d^{\prime}$ divides both $B_{k,m}$ and $B_{k,mq+1}B_{k,r}$ which means $d^{\prime}|d$, thus $d=d^{\prime.}$ Now, \begin{align*}\gcd(B_{k,n},B_{k,m})&=\gcd(B_{k,m},B_{k,r})\\
&=B_{k,\gcd(m,r)}& \text{by induction}\\
&=B_{k,\gcd(m,mq+r)}\\
&=B_{k,\gcd(m,n)},\end{align*} 
as desired.
\end{proof}

\end{document}